\documentclass[11pt]{amsart}
\usepackage{amssymb,hyperref}
\usepackage{graphicx}
\usepackage[matrix,arrow]{xy}
\newcommand{\PP}{\mathbb P}

\newcommand{\CC}{{\mathbb C}}
\newcommand{\NN}{{\mathbb N}}
\newcommand{\ZZ}{{\mathbb Z}}
\newcommand{\QQ}{{\mathbb Q}}
\newcommand{\RR}{{\mathbb R}}

\newcommand{\Num}{\mathrm{Num}}
\newcommand{\wt}{\mathrm{wt}}

\newcommand{\IK}{{{\rm I}}}
\newcommand{\II}{{\mathop{\rm II}}}
\newcommand{\III}{{\mathop{\rm III}}}
\newcommand{\IV}{{\mathop{\rm IV}}}

\newcommand{\tD}{2h}
\newcommand{\jD}{h}
\newcommand{\polD}{H}

\DeclareMathOperator{\rank}{rank}

\newcommand{\WW}{Y}  

\numberwithin{equation}{section}
\begin{document}

\title{12 rational curves on Enriques surfaces}

\author{S\l awomir Rams}
\address{Institute of Mathematics, Jagiellonian University, 
ul. {\L}ojasiewicza 6,  30-348 Krak\'ow, Poland} 
\email{slawomir.rams@uj.edu.pl}

\author{Matthias Sch\"utt}
\address{Institut f\"ur Algebraische Geometrie, Leibniz Universit\"at
  Hannover, Welfengarten 1, 30167 Hannover, Germany}

    \address{Riemann Center for Geometry and Physics, Leibniz Universit\"at
  Hannover, Appelstrasse 2, 30167 Hannover, Germany}

\email{schuett@math.uni-hannover.de}

\date{April 1, 2021}
\thanks{Research partially supported by the National Science Centre, Poland, Opus  grant 
no.\ 2017/25/B/ST1/00853
(S.\ Rams)}

\begin{abstract}
Given $d\in\NN$, we prove that  
any polarized Enriques surface (over any field $k$ of characteristic $p \neq 2$ or with a smooth K3 cover) 
of degree greater than $12d^2$
contains at most 12 rational curves of degree at most $d$. For $d>2$ we construct examples of Enriques surfaces of high degree 
that contain exactly 12 rational degree-$d$ curves. 
%
\end{abstract}

\maketitle

\newcommand{\XXd}{X_{d}}
\newcommand{\XXf}{X_{4}}
\newcommand{\XXp}{X_{5}}
\newcommand{\mF}{\mathcal F}
\newcommand{\MW}{\mathop{\mathrm{MW}}}
\newcommand{\mL}{\mathcal L}
\newcommand{\mR}{\mathcal R}
\newcommand{\Ruledeight}{S_{11}}
\newcommand{\Ruledfour}{S_{4}}
\newcommand{\DivisorRest}{\mathfrak Rest}
\newcommand{\Pl}{\Pi}
\newcommand{\reg}{\operatorname{reg}}

\theoremstyle{remark}
\newtheorem{obs}{Observation}[section]
\newtheorem{rem}[obs]{Remark}
\newtheorem{example}[obs]{Example}
\newtheorem{ex}[obs]{Example}
\newtheorem{conv}[obs]{Convention}
\theoremstyle{definition}
\newtheorem{Definition}[obs]{Definition}
\theoremstyle{plain}
\newtheorem{prop}[obs]{Proposition}
\newtheorem{theo}[obs]{Theorem}
\newtheorem{Theorem}[obs]{Theorem}
\newtheorem{lemm}[obs]{Lemma}
\newtheorem{crit}[obs]{Criterion}
\newtheorem{claim}[obs]{Claim}
\newtheorem{Fact}[obs]{Fact}
\newtheorem{cor}[obs]{Corollary}
\newtheorem{assumption}[obs]{Assumption}

\newcommand{\ux}{\underline{x}}
\newcommand{\ud}{\underline{d}}
\newcommand{\ue}{\underline{e}}
\newcommand{\mmS}{{\mathcal S}}
\newcommand{\mmP}{{\mathcal P}}
\newcommand{\nlines}{\mbox{\texttt l}(\XXp)}
\newcommand{\ii}{\operatorname{i}}

\newcommand{\nonlinflec}{{\mathcal D}}
\newcommand{\linflec}{{\mathcal L}}
\newcommand{\flec}{{\mathcal F}}


\section{Introduction}
\label{intro}

It is known by the work of Barth--Peters, Nikulin, and Kondo
that any Enriques surface $Y$ over $\CC$ contains infinitely many (possibly singular) rational curves.
In fact, as soon as there is a single smooth rational curve, then there are infinitely many of them
outside 7 specific cases -- the types I--VII due to Nikulin \cite{Nik} and Kondo \cite{Kondo}
which will also be relevant for our considerations.
Here we restrict the problem to (again, possibly singular) rational curves of small degree
relative to a given polarization $H$ of $Y$.
For $d\in\NN$, let
\[
r_d := r_d(Y) := \#\{
\text{rational curves } C\subset Y \text{ with } \deg(C)= d = C.H
\}
\]
and
\[
S_d := r_1+\hdots+r_d = \#\{\text{rational curves } C\subset Y \text{ with } \deg(C)\leq d\},
\]
We work over an algebraically closed field $k$,
but our main result  is almost independent of the characteristic in the sense
that we only require that the K3 cover of the Enriques surface $Y$ is smooth 
(i.e.\ we exclude the types of classical and supersingular Enriques surfaces in characteristic $2$):

\begin{Theorem} \label{thm}
Let $d\in\NN$.
Assume that $Y$ is an Enriques surface with smooth K3 cover.
Fix a polarization $H$ on $Y$ such that $H^2>12d^2$.
Then 
\[
S_d\leq 12.
\]
\end{Theorem}

The methods to prove this result 
build on  those developed for K3 surfaces  in \cite{RS24},
but remarkably, the result is substantially stronger in the sense that the bound $H^2>12d^2$
cannot be improved much (outside characteristics $2,5$ at least, see Proposition \ref{prop:13}) for $d>1$.
Here the often delicate geometry of Enriques surfaces acts very much to our advantage.
We also show by engineering explicit surfaces that the bound is attained as follows:

\begin{prop}
\label{prop}

\begin{enumerate}
\item[(i)]
For any even $d\geq 6$ and any $h\geq 3d$, there is a $2h$-polarized Enriques surface
with $r_d=12$.
\item[(ii)]
Assume that char$(k)\neq 2$. For any $d\geq 3$ and any $h\geq 9$ with $d\mid h$, there is a $2h$-polarized Enriques surface
containing 12 smooth rational curves of degree $d$. 
\end{enumerate}
\end{prop}
 Over fields of characteristic $2$ the statement of Proposition \ref{prop}~(ii)
no longer holds, see Proposition \ref{prop:p=2}.

The cases of smooth rational curves of degree $1$ and $2$, i.e.\ lines and conics,
are more delicate than the above and will  be treated elsewhere.

%
%

%
%

\section{Set-up}

Throughout this note, we consider polarized Enriques surfaces of degree $2h$,
 i.e.\ pairs $(Y,H)$ where $Y$ is an Enriques surface  over $k$
 and  $H$ 
is a very ample divisor of square $\polD^2=\tD$.
If char$(k)=2$, then we assume that $Y$ has a smooth K3 cover $X$
(these Enriques surfaces are often called singular).
Then, the linear system $|\polD|$ defines an embedding
\[
\varphi_{|\polD|}: \quad Y \hookrightarrow Y_{\polD} \subset \PP^{\jD}
\]
which is an isomorphism onto its image.

It is well-known that each very ample divisor on an Enriques surface satisfies the inequality $H^2 \geq 10$. 
For $\mbox{char}(k) \neq 2$ the question whether  
 a given divisor $H$ on $Y$ 
is very ample
can be answered by the methods developed in \cite{cossec-projective}. The discussion of the general case  
can be found in \cite{Cossec-Dolgachev-Liedtke}.
For the convenience of the reader we recall the (well-known) result we need below.

\begin{crit}
\label{crit}
{\rm (\cite[Remark~2.4.19]{Cossec-Dolgachev-Liedtke})} Let $H$  be a big and nef divisor on an Enriques surface $Y$. 
Then $H$ is very ample if and only if
\begin{equation} \label{eq-crit-very-ample}
H.D \geq 3 \mbox{ for every half-pencil } D \mbox{ of a genus 1 fibration on } Y
\end{equation}
and $H.E > 0$ for every $(-2)$-curve $E \subset Y$.
\end{crit}

\begin{proof} 
By \cite[Lemma~2.4.10]{Cossec-Dolgachev-Liedtke} there exists a half-pencil $D'$ of a genus 1 fibration on $Y$
such that the equality $\Phi(H)= H.D'$ holds.
Thus Criterion~\ref{crit} is   \cite[Remark~2.4.19]{Cossec-Dolgachev-Liedtke}.
\end{proof}

%
%
%

\section{Basics}
\label{s:prep}

Given a polarized Enriques surface $Y$ we follow  \cite{RS24} and  consider the set
\[
\Gamma = \{\text{rational curves } C\subset \WW \text{ of degree } C.H\leq d \}
\]
which we  interpret as a graph without loops
with (possibly multiple) edges 
corresponding to the intersection points of rational curves  $C, C'\in\Gamma$.
Here each vertex $C\in\Gamma$ comes attached with two values,
the degree $C.H$ and the square $C^2$.
Together the rational curves generate the formal group 
$$M:=\ZZ\Gamma\subset\mathrm{Div}(\WW),
$$
equipped with the intersection pairing extending that on the vertices.
Clearly $M$ may be degenerate 
(i.e.\ $\ker(M)\neq 0$). By the Hodge index theorem we have the following three cases:
\begin{enumerate}
\item
elliptic -- $M\otimes\RR$ is negative-definite ($\ker(M)=0$);
\item
parabolic -- $M\otimes\RR$ is negative semi-definite, but not elliptic ($\ker(M)\neq 0$);
\item
hyperbolic -- $M\otimes\RR$ has a one-dimensional positive-definite subspace
and none of greater dimension.
\end{enumerate}

The elliptic case and the parabolic case can be studied exactly along the lines of \cite[\S 4,5]{RS24}
-- with the extra benefit that there are no quasi-elliptic fibrations complicating the analysis in small characteristic
(cf.\ \cite[Thm~4.9.3]{Cossec-Dolgachev-Liedtke}).  
For brevity, we just sketch the arguments:

\begin{lemm}
\label{lem:ortho}
\begin{enumerate}
\item[(i)]
If $M$ is elliptic, then it is an orthogonal finite sum of Dynkin diagrams (ADE-type).
\item[(ii)]
If $M$ is parabolic, then it is an orthogonal finite sum of Dynkin diagrams
and at least one isotropic vertex or 
extended Dynkin diagram ($\tilde A\tilde D\tilde E$-type).
\end{enumerate}
\end{lemm}
\begin{proof} 
(i) If $M$ is elliptic, then $C^2=-2$ for all $C\in \Gamma$.
Moreover,
\[
C.C'\leq 1\;\;\, \forall \, C, C'\in\Gamma,
\]
for otherwise $(C+C')^2\geq 0$.
Hence $M$ is an orthogonal  sum of Dynkin diagrams;
since $M$ embeds into the hyperbolic  lattice $\Num(Y)$ of rank $10$ (cf.\ \cite{BM76}),
the rank of $M$ is bounded by $\rho(Y)-1=9$.
Thus the orthogonal sum is finite  as claimed (and there are only finitely many possible configurations).

\noindent
(ii) If $M$ is parabolic, then we directly obtain a constraint on self-intersection
\begin{eqnarray}
\label{eq:02}
C^2 \in\{0, -2\} \;\;\; \forall \, C\in\Gamma.
\end{eqnarray}
Similarly, intersection numbers of pairs of curves are restricted as follows. 
For any isotropic $C\in\Gamma$, we have
\begin{eqnarray}
\label{eq:isotropic}
C.C'=0 \;\;\; \forall \, C' \in\Gamma,
\end{eqnarray}
for else $(2C+C')^2>0$. The same statements holds
for the intersection of $C' \in \Gamma$ with any isotropic divisor $D$ supported on $\Gamma$,
and for any two isotropic divisors supported on $\Gamma$.
For the same reason, the following inequality holds
\begin{eqnarray}
\label{eq:<=2}
C.C' \leq 2 \;\;\; \forall \; C, C'\in\Gamma\;\;  \text{ with } C^2=C'^2=-2.
\end{eqnarray}
Here equality makes $(C+C')$ isotropic and thus orthogonal to all other curves in $\Gamma$
(so $C, C'$ generate the extended Dynkin diagram $\tilde A_1$).
It follows that $\Gamma$ consists of pairewise disjoint isotropic curves (i.e.\ curves of arithmetic genus $p_a=1$,
a priori possibly infinite in number)
and standard or extended Dynkin diagrams (again finite in number for rank reasons).
\end{proof}

\begin{cor}
\label{cor:para}
\begin{enumerate}
\item[(i)]
If $M$ is elliptic, then $\#\Gamma = \rank M \leq 9$.
\item[(ii)]
If $M$ is parabolic, then $\#\Gamma \leq 12$.
\end{enumerate}
\end{cor}

\begin{proof}
(i) If $M$ is elliptic, then $\#\Gamma = $ rank$(M)\leq 9$ as in the proof of the part (i) of Lemma \ref{lem:ortho}.

\noindent
(ii) If $M$ is parabolic, then Lemma \ref{lem:ortho} 
implies the  existence of
a divisor $D$ of Kodaira type supported on $\Gamma$,
i.e.\ a nodal or cuspidal cubic or, in the case of an extended Dynkin diagram,
a configuration of $(-2)$ curves fitting a singular fibre
of an elliptic fibration 
 as classified by Kodaira \cite{Kodaira} and Tate \cite{Tate}.
Then $|D|$ or $|2D|$ induces a genus one fibration
\begin{eqnarray}
\label{eq:fibr}
Y \to \PP^1
\end{eqnarray}
with $D$ as a (multiple) fibre.
By construction,
$\Gamma$ is the set of rational fiber components of \eqref{eq:fibr} that have degree at most $d$.
The fibration formula for the Euler--Poincar\'e characteristic $e(Y)=12$
yields the claimed bound.
\end{proof}

The corollary also provides us with a recipe for producing Enriques surfaces
of arbitrary polarization
with 12 rational curves of small degree by arranging for them to be fibre components
of a suitable genus one fibration.
This will be exploited in Sections \ref{ss:even_d} and \ref{ss:all_d}.


\section{Divisors of Kodaira type}

Drawing closer to the proof of Theorem \ref{thm},
we assume throughout this and the next four sections
 that
\begin{eqnarray}
\label{eq:H^2}
H^2 >12d^2.
\end{eqnarray}
As in \cite[\S 7]{RS24}, this implies that the restrictions
\eqref{eq:02}, \eqref{eq:isotropic} and \eqref{eq:<=2}
continue to hold true.
Moreover, the analogue of \eqref{eq:isotropic},
%
%
%
\begin{eqnarray}
\label{eq:D-6d}
D.C'=0 \;\;\; \forall \text{ isotropic } C' \in\Gamma,
\end{eqnarray}
holds for any isotropic effective divisor $D$ with $\deg(D)\leq 6d$ supported on $\Gamma$.
To see this, just 
set $c:=D.C', d':=\deg(C'), d'':=\deg(D)$,
write down the Gram matrix of $H, C', D$, and compute its determinant
\[
-c(cH^2-2d''d') \, , 
\]
which is negative for $c\geq 1$.
But then the lattice generated by  $H, C', D$ cannot be hyperbolic, contradiction.

The following property will be instrumental for all arguments to follow.

\begin{lemm}
\label{lem:Kodaira_type}
If $\Gamma$ is not elliptic, then it supports a divisor of Kodaira type.
\end{lemm}

\begin{proof}
If $\Gamma$ contains a curve $C$ with $C^2=0$, then we're done,
so by \eqref{eq:02} we may assume that $\Gamma$ consists of $(-2)$-curves.
Thus we may take a maximal elliptic subconfiguration $\Gamma'$ of $\Gamma$
(an orthogonal sum of Dynkin diagrams)
and attach any curve in $\Gamma \setminus \Gamma'$ to it.
By assumption, the resulting configuration is no longer  elliptic.
We claim that it supports an isotropic vector.
Indeed, using the restriction \eqref{eq:<=2},
verifying the claim amounts to a simple case-by-case analysis starting 
from any Dynkin diagram involved in $\Gamma'$.
\end{proof}

We continue by explaining how
Lemma~\ref{lem:Kodaira_type} combined  with Corollary \ref{cor:para},
gives way to an ineffective proof of Theorem \ref{thm}
(i.e.\ for $H^2\gg 0$ -- we will use this later for Proposition \ref{prop:p=2}). 
This builds on the concept of instrinsic polarization in the terminology of \cite{degt}.
For any hyperbolic subgraph $\Gamma'\subset\Gamma$, 
the instrinsic polarization is defined as 
the $\QQ$-divisor $H_{\Gamma'}\in\QQ\Gamma'$ determined 
(if it exists, otherwise we get a contradiction) by the degree conditions
\begin{equation} \label{eq-defHintrinsic}
C.H_{\Gamma'} = C.H \;\;\; \forall\, C\in\Gamma'.
\end{equation}
As in  \cite[Prop.\ 6.2]{RS24}, the polarization $H$ on $Y$ crucially satisfies the condition
\begin{eqnarray}
\label{eq:intrinsic}
H^2 \leq H_{\Gamma'}^2.
\end{eqnarray}

\begin{lemm}
\label{lem:para}
If $\#\Gamma>9$ and $H^2\gg 0$,
then $\Gamma$ is parabolic.
\end{lemm}

\begin{proof}
By Corollary \ref{cor:para} (i), $\Gamma$ cannot be elliptic,
so we assume it to be hyperbolic.
By Lemma \ref{lem:Kodaira_type},
$\Gamma$ contains a parabolic subgraph $\Gamma_0$
which may be extended by a single curve $C\in\Gamma$ 
to a hyperbolic subgraph $\Gamma'\subset\Gamma$.
Presently, there are only finitely many possible configurations for $\Gamma_0$;
by the above restrictions \eqref{eq:02}--\eqref{eq:<=2},
the same holds for $\Gamma'$. Finally, by  \eqref{eq:intrinsic},
the maximum of all self-intersections $H_{\Gamma'}^2$ gives an upper bound for $H^2$.
\end{proof}

\begin{proof}[Ineffective proof of Theorem \ref{thm}]
If $\#\Gamma>9$ and $H^2\gg 0$,
then  $\Gamma$ is parabolic by Lemma \ref{lem:para}.
Thus Corollary \ref{cor:para} (ii) shows that $\#\Gamma\leq 12$.
\end{proof}

\begin{rem} \label{rem:inverse-gram}
To obtain a lower bound for $H^2$ to rule out $\Gamma$ being hyperbolic,
we will refine the idea that 
appeared already in the proof of Lemma~\ref{lem:para} (see also \cite{RS24}). 

We  assume that $\Gamma$ is hyperbolic. 
We choose  a basis $\Gamma_0$ of $\Gamma/\ker$  
with Gram matrix  $G$. 
The intrinsic polarization can be expressed as
\[
H_0=G^{-1}\vec d \, , 
\]
where the coordinates of $\vec d$ are the degrees of the elements of the basis $\Gamma_0$. Thus 
estimating $H_0^2$ reduces to an optimization problem. But all degrees are positive, so we obtain an easy  bound
in terms of the entries $g_{ij}$ of $G^{-1}$ by
\begin{eqnarray}
\label{eq:H_0^2}
H_0^2 \leq \sum_{i,j} \max(0,g_{ij}) d^2.
\end{eqnarray}
This bound is attained when all $g_{ij}$ are non-negative by all curves in $\Gamma_0$ having degree $d$.
In the presence of negative entries, 
the bound can be  improved by arranging for a  decomposition
$G^{-1} = G_{0} + G_{+}$, with $G_{+}$ negative semi-definite 
containing in its kernel the vector with all entries the same
(cf.\  \cite[Lemma~9.1]{RS24}).
 \end{rem}


\section{Hyperbolic case -- preparations} \label{hyp-prep}

In view of Corollary \ref{cor:para}, it remains to study the case when $M$ is hyperbolic 
in order to work out an effective
 proof of Theorem \ref{thm}. 

In this section we maintain the assumption \eqref{eq:H^2}.
As in \cite{RS24}, for a divisor $D= \sum_{i} n_{i} C_{i}$ of Kodaira type,  we define the weight by the equality $\mbox{wt}(D) = \sum_i n_{i}$.
In particular, if $D$ is supported on $\Gamma$, we have $\deg(D)\leq \wt(D)d$.


\begin{lemm}
\label{lem:types}
Assume that $\Gamma$ is hyperbolic.
Then 
\begin{eqnarray}
\label{eq:C}
C^2=-2\;\; \;\forall\, C\in\Gamma,  \;\; \text{ and } \;\;
C.C'\leq 1 \;\;\; \forall\,C, C'\in\Gamma.
\end{eqnarray}
In particular, there are no divisors of Kodaira type $\IK_1, \IK_2, \II, \III, \IV$ supported on $\Gamma$. 
Moreover, if a divisor of type $\IK_3$ or $\IK_4$ is supported on $\Gamma$,
then it is a half-pencil of a genus one fibration such that $\Gamma$ contains none of its  multisections  of index $>2$. 
\end{lemm}

\begin{proof} Let  $D$ be a divisor of Kodaira type, such that all components of its support belong to $\Gamma$.
Since $\Gamma$ is hyperbolic,  there is a curve $C\in\Gamma$ serving as a multisection of 
the genus one fibration \eqref{eq:fibr}.

We let $r:=C.D$; note that, for a general fibre $F$, we have
$C.F=r$ or $2r$, depending on whether $D$ is a fiber or a half-pencil.
Consider the sublattice 
$$
L = \langle C,D\rangle\subset\Num(Y).
$$
We consider the intrinsic polarisation
\[
H_L\in L\otimes\QQ \;\;\; \text{ determined by  } \;\; C.H_L = C.H, \;\; D.H_L = D.H,
\]
that satisfies the inequality $H^2 \leq H_L^2$.
Solving the above linear system for $H_L$, we use \eqref{eq:H^2} to obtain the following inequalities:
\begin{equation} \label{eq-useful-ineq}
12d^2 < H_L^2 \leq 2\deg(D) (r\deg(C)+\deg(D))/r^2 \leq 2\wt(D) (r+\wt(D))d^2/r^2.
\end{equation}
This readily implies that $\wt(D)>2$, ruling out fibres of Kodaira type $\IK_1, \IK_2, \II, \III$
to be supported on $\Gamma$. In particular, combined with \eqref{eq:02}, it shows that any curve in $\Gamma$ is nodal (i.e.\ smooth rational). 
Furthermore,  $C.C' \neq 2$ for any $C, C' \in \Gamma$ , so \eqref{eq:C} follows from \eqref{eq:<=2}. 

Finally, if $D$ has weight at most $4$ (or more generally degree at most $4d$), 
then \eqref{eq-useful-ineq} yields $r=1$, making $D$ a half-pencil 
and any multisection $C\in\Gamma$ a bisection.
This applies to Kodaira types $\IK_3, \IK_4$ and rules out $\IV$
(since multiple fibres are multiplicative).
\end{proof}

In \cite{Kondo}, Kondo pioneered a construction of (complex) Enriques surfaces
by using suitable disjoint sections on elliptic K3 surfaces.
This was later generalized in \cite{HS} as we shall use in Section \ref{ss:all_d}.
Here we note the following useful consequence:

\begin{cor}
If $\Gamma$ is hyperbolic and supports a divisor of Kodaira type $\IK_3$ or $\IK_4$, then
$Y$ arises from Kondo's construction.
\end{cor}

\begin{proof}
We have just seen that any multisection in $\Gamma$ is a bisection.
This splits on the K3 cover into two disjoint sections.
Hence the reasoning from \cite{S-Q-hom}, \cite{S-Q_l}, \cite{S-Q_2} (or \cite{Martin}) applies
to show that, independent of the characteristic, $Y$ arises from Kondo's construction.
\end{proof}

\section{Hyperbolic case -- reduction}

The following lemma  will substantially simplify our analysis of possible hyperbolic graphs $\Gamma$.
As stated before, in this section, we continue to make the assumption  \eqref{eq:H^2}, so we can apply the results from the previous two sections.

\begin{lemm}
\label{lem:red1}
If $\Gamma$ is hyperbolic, then it supports a divisor $D$ of Kodaira type $\IK_n \,(n\geq 3)$
which is a half-pencil of a genus one fibration on $Y$. Moreover, one can choose $D$ such that 
all multisections of $|2D|$ that belong to $\Gamma$ are bisections.
\end{lemm}

\begin{proof}
If $\Gamma$ supports a half-pencil, then it has the claimed type,
so let us just assume to the contrary that $\Gamma$ supports no half-pencil.

By Lemma \ref{lem:Kodaira_type}, there is a divisor $D$ of Kodaira type supported on $\Gamma$.
Consider $D$ together with a multisection $C\in\Gamma$ as in \S\ref{hyp-prep}.
Since $D$ is not a half-pencil, we have $C.D\geq 2$.
Recall from \eqref{eq:C} that $\Theta.C\leq 1$ for any component $\Theta$ of $D$ and the multisection $C$ is nodal.

If $D$ has type $\IK_m^*$, then either $C$ meets at least two components of $D$,
so there is a cycle (type $\IK_n$) supported on $\Gamma$, or it meets just a single double fibre component.
But then the multisection meets the fiber component transversally, so $C$ must be a bisection. 
Moreover, on the K3 cover $C$ splits into two (disjoint) sections, none of which may meet
the double component of an $\IK_m^*$ fiber of the fibration induced on the K3 cover, contradiction.

The same reasoning applies to $D$ of Kodaira type $\IV^*$. Indeed, the intersection $D.C$ must be even, so 
$C$ cannot meet $D$ in exactly one point on the triple component.

For type $\III^*$, the multisection $C$ could a priori also meet just the 4-fold fibre component $\Theta$,
but then $\Gamma$ would also support a divisor of type $\IK_0^*$, and we conclude as before.
An analogous argument applies to  type $\II^*$.

Finally, if there is a divisor of type $\IK_n$ supported on $\Gamma$ which is not a half-pencil,
then it connects with $C$ through at least two distinct fibre components;
this yields a cycle of length at most $\lfloor\frac n2\rfloor+2$.
Iterating this procedure as long we  do not get a half-pencil, 
we obtain a divisor of type $\IK_3$ or $\IK_4$
supported on $\Gamma$.
This is a half-pencil by Lemma \ref{lem:types}, contradiction, which completes the proof of the first claim of the lemma.

Let $D$ be a divisor of Kodaira type $\IK_n$ $(3 \leq n \leq 9)$ supported on
$\Gamma$ such that $|2D|$ induces a genus one fibration. 
If  $\Gamma$ contained a multisection of $|2D|$ of index $> 2$, then the latter could be used to produce a cycle of type $\IK_3$ or $\IK_4$ with all components in $\Gamma$.
By Lemma \ref{lem:types} $\Gamma$ contains only bisections of the fibration such a cycle defines, so we can find a divisor $D$ of Kodaira type $\IK_n$ supported on $\Gamma$  
such that all multisections in $\Gamma$ are in fact bisections of $|2D|$.
\end{proof}


\section{Hyperbolic case with at most 3 bisections}
\label{ss:3}

Let $\#\Gamma>12$.
In this section we rule out the possibility that $\Gamma$ contains at most three bisections of the genus one fibration $|2D|$
(that exists by Lemma~\ref{lem:red1}) as soon as  \eqref{eq:H^2} holds.

By Cor.~\ref{cor:para} the graph $\Gamma$ is hyperbolic.
If there are at most three bisections of $|2D|$  in $\Gamma$,
then there are at least ten fibre components supported on $\Gamma$.
As soon as  \eqref{eq:H^2} holds, this turns out to be very restrictive, since fibers of type $\IK_1, \IK_2$ cannot be supported on $\Gamma$
by Lemma \ref{lem:types}
and types $\IK_3, \IK_4$ are automatically multiple. 

Recall that the Jacobian fibration of $|2D|$ is a rational elliptic surface.
Naturally it shares the  same singular fibers with $Y$,
except that on $\WW$,
 smooth or semi-stable fibers (Kodaira type $I_n, n\geq 0$) may come with multiplicity two.
The classification of rational elliptic surfaces of \cite{OS} 
 and Lemma~\ref{lem:types} then give the  six configurations \ref{conf-2I5} -- \ref{ss:VII} below,
each with exactly 3 bisections supported on $\Gamma$.

Let $B \in \Gamma$ be a bisection of $|2D|$. Since $B$ is a nodal curve, it  splits into two disjoint sections $O, P$
of the induced elliptic fibration on the  K3 cover $X$
and one can use the theory of Mordell--Weil lattices \cite{MWL},
and the fact that $B$ meets no curve in $\Gamma$ with multiplicity greater than one,
to compute (an upper bound for) the height of $P$.
We also use the classification of extremal rational elliptic surfaces in \cite{Lang1}, \cite{Lang2}
and the geometry  of Enriques surfaces with finite automorphism groups analyzed in \cite{Kondo}, \cite{Martin}.

\subsection{
Two $\IK_5$} \label{conf-2I5}

In this case we have
$h(P)\leq 4-\frac 52 - 2\cdot\frac 45<0$, contradiction.

\subsection{
$\IK_4$ and $\IK_1^*$} \label{conf-I5I1s}

We obtain
$h(P)\leq 4-2-2\cdot 1 =0$, so $P$ necessarily is $2$-torsion,
and $Y$ has Kondo's type II with exactly 12 smooth rational curves
by \cite{Kondo}, \cite{Martin}, so $\#\Gamma\leq 12$, contradiction.


\subsection{ 
$\IK_3$ and $\IV^*$} \label{conf-I3IVs}

Here $h(P)=4-\frac 32-2\cdot\frac 43<0$, contradiction.

\subsection{
Two $\IK_3$ and two $A_2$ embedding into $\IK_3$ each} \label{conf-4I3s}
(outside characteristic $2$, since then there can only be one multiple fibre,
and characteristic $3$, because the fibration ceases to exist there)

In this case $h(P)= 4-2\cdot\frac 32=1$, 
i.e.\ $B$ meets the components of the unramified $\IK_3$ fibres not contained in $\Gamma$,
since other configurations would cause $h(P)<0$.
It follows that the K3 cover has $\rho(X)>20$, so $\rho(X)=22$ by \cite{Artin}.
But then the lattice $L$ generated by fibre components, torsion sections and $P$ embeds into some
supersingular K3 lattice $\Lambda_{p,\sigma}$ which is $p$-elementary.
However, $L$ is primitive (since else there would be a section $Q$ of height $h(Q)=h(P)/9=1/9$
which is impossible with the given fibre types),
and one computes that $L^\vee/L$ contains a subgroup isomorphic to $(\ZZ/3\ZZ)^2$.
This prevents $L$ from embedding into $\Lambda_{p,\sigma}$
(cf.\ \cite[Thm.\ 6.1]{KS}).


\subsection{
Two $\IK_4$ and two $A_1$ embedding into $\IK_2$ each} \label{conf-2I42I2}
(outside characteristic $2$, since this surface ceases to exist there)

We have $h(P)\leq 4-2\cdot 2=0$, so $P$ is $2$-torsion
and $Y$ has Kondo's type III.
This exists outside characteristics $2,5$
with exactly 20 smooth rational curves
by \cite{Kondo}, \cite{Martin}.
Fixing two $\IK_4$ fibres and a bisection $B$ to be contained in $\Gamma$,
$B$ determines the precise components $\Theta, \Theta'\in\Gamma$ of the $\IK_2$ fibres, 
since it meets the respective other component with multiplicity two
(which is thus not contained in $\Gamma$ by Lemma \ref{lem:types}).
But then all other multisections are ruled out, again by Lemma \ref{lem:types}, 
since each meets $\Theta$ or $\Theta'$ with multiplicity two,
so $\# \Gamma=11$, contradiction.


\subsection{
$\IK_3, \IK_6$ and $A_1$ embedding into $\IK_2$ (or into $\III$ in characteristic $3$)}
\label{ss:VII}
(again outside characteristic $2$, since there $\IK_3$ degenerates into $\IV$
which cannot be multiple)

In this case,  we have $h(P)\leq 4-\frac 32-2\cdot \frac 56 = \frac 56$, so this case is not ruled out as the other ones.
However, any other configuration would yield $h(P)<0$,
so the configuration is unique, and we deduce that $Y$ is of Kondo's type VII.

\vspace*{2ex}
Until now, we have shown the following lemma.
\begin{lemm}
\label{lem:red1-3sect}
Let  $\#\Gamma>12$ and $H^2 >12d^2$.   If $\Gamma$ contains at least 10 components of singular fibers of the genus one fibration $|2D|$ (cf. Lemma~\ref{lem:red1}),
 then it contains exactly 10 components and they  
 form   the configuration {\rm \ref{ss:VII}}. 
\end{lemm}

To complete the analysis of the case with 3 or less bisections contained in $\Gamma$
we refrain to an optimization approach.
To this end, we assume that the ten fiber components form  the configuration \ref {ss:VII} and fix a $\QQ$-basis of $\Num(Y)$ 
consisting of a bisection $B$, $\IK_3$ components, five components of $\IK_6$ (only one 
of which meets $B$),
and the $\IK_2$ component off $B$.

\begin{lemm} \label{lem:conf76}
If the above 10 rational curves and the remaining component of $\IK_6$ are all in $\Gamma$,
then $H^2 < 12d^2.$
\end{lemm}

\begin{proof}
The above curves generate $\Num(Y)$, so their degrees determine the polarization $H$
and there is no need to consider the intrinsic polarization.
Computing an upper bound for $H^2$
amounts to optimizing the quadratic form given by the inverse of the Gram matrix $G$
over (the $\QQ$-points of) the set $(0,1]^{10}$ (taking into account that $\deg(\IK_6)=2\deg(\IK_3)$).
Thanks to the specific shape of $G^{-1}$, this can be broken down into
two separate optimization problems -- one in 3 variables and one in 4 variables.
These are few enough variables to handle all the boundary components involved.
Interlacing  the two resulting partial upper bounds, one obtains that the quadratic form 
always evaluates smaller than $12$ as desired.
\end{proof}

This completes the analysis of the case with 3 or less bisections contained in $\Gamma$.

\begin{rem}
We will see in Proposition \ref{prop:13}
that Lemma \ref{lem:conf76} is close to sharp in the sense
that the above Enriques surface contains 13 smooth rational curves of degree $\leq d$ 
with respect to a polarization of degree $12d-2$.
\end{rem}

\section{Proof of Theorem~\ref{thm} }

In this section we finally prove Theorem~\ref{thm}.
The proof is preceded by several lemmas.
At first we reduce our analysis of Kodaira divisors to two fibre types only.

\begin{lemm}
\label{lem:red2}
If  
$H^2 >12d^2$ and $\#\Gamma>12$, then $\Gamma$ supports 
a divisor of type either $\IK_3$ or $\IK_4$.
\end{lemm}

\begin{proof}
By
Lemma \ref{lem:red1}, there is a half-pencil of type $\IK_n$ with all its components in $\Gamma$, such that $\Gamma$ contains only 
fiber components and bisections of the genus one fibration it defines. Moreover,  Lemma~\ref{lem:red1-3sect} combined with Lemma~\ref{lem:conf76} shows that 
 $\Gamma$ supports at least 4 such bisections.

If two of the bisections meet the same fibre component,
then either they are adjacent, so we get a $\IK_3$ divisor,
or they are perpendicular and we get a $\IK_0^*$ divisor
which is met by some fibre component with multiplicity one, contradiction.
Hence all bisection meet different fibre components.
If $n<8$, then there are two bisections meeting adjacent fibre components.
Thus we get a $\IK_4$ divisor (if the bisections intersect)
or a $\IK_1^*$ divisor (if the bisections are disjoint), again intersected 
by some fibre component with multiplicity one.
If $n=8$ or $9$, the same reasoning applies,
but the divisors could also have types $\IK_2^*$ (with the same contradiction)
and $\IK_5$ -- which would be multiple,
so we can revisit the case $n<8$
to conclude the proof.
\end{proof}

Now we are in the position to rule out the existence of a triangle (i.e.\ a $\IK_3$ fiber) in $\Gamma$.
\begin{lemm}
\label{lem:I_3}
If
$H^2 >12d^2$ and  $\#\Gamma>12$, then $\Gamma$ supports 
no divisor of type $\IK_3$.
\end{lemm}

\begin{proof}
Assume to the contrary that $\Gamma$ supports a triangle. 
Since $\# \Gamma > 12$, the graph $\Gamma$ is hyperbolic, so we can apply Lemma~\ref{lem:types} 
to show that 
the $\IK_3$-configuration is a half-pencil of a genus one fibration
such that $\Gamma$ contains only its fiber components and bisections.  
Moreover,  by Lemmas~\ref{lem:red1-3sect}, \ref{lem:conf76} $\Gamma$ must contain at least four
bisections. Thus we can assume that $\Gamma$ supports a divisor of type $\IK_3$
with 4 bisections. We claim that 
\begin{equation} \label{eq-ineqI3}
H^2\leq 11\frac 56\,d^2 .
\end{equation}
Indeed, the proof of \eqref{eq-ineqI3} amounts to  bounding the degree of the intrinsic polarization for some test (sub)configurations.
The bound is obtained from the entries of the inverse of the Gram matrix of the (sub)configuration in question (see Remark~\ref{rem:inverse-gram}).
If there are two adjacent bisections,
then the rank five lattice $L$ generated by the $\IK_3$ fibre and these bisections gives
$$
H_L^2\leq 
\begin{cases} 9\frac 16 & \text{if the bisections meet the same fibre component,}\\
10\frac 38 & \text{if the bisections meet  different fibre components.}
\end{cases}
$$
Otherwise all bisections are perpendicular, and a quick analysis 
gives the inequality \eqref{eq-ineqI3} for the intrinsic polarization in each possible case.
(In the case when each fiber component is met by at least one  bisection,  
one can split off a negative semi-definite matrix as in Remark \ref{rem:inverse-gram}
to compensate for negative entries in $G^{-1}$.) 

Obviously \eqref{eq-ineqI3} yields the desired contradiction and completes the proof.
\end{proof}

In particular, we have shown that $\Gamma$ contains a quadrangle (i.e.\ an $\IK_4$ fiber)
 as soon as  $\#\Gamma>12$ and the degree of the polarization is high enough (i.e.\ \eqref{eq:H^2} holds).  
Recall that, by Lemma~\ref{lem:types}, the system $|2\IK_4|$ endows the Enriques surface $Y$ with a genus one fibration. 

\begin{lemm}
\label{lem:I_4-5bisections}
If $H^2 >12d^2$ and  $\#\Gamma>12$, then $\Gamma$ supports at most $4$ bisections of the fibration $|2\IK_4|$.
\end{lemm}
\begin{proof} 
%
Observe that any two bisections meeting the same fibre component
are disjoint, for otherwise they would form a $\IK_3$ divisor, which we have covered already in Lemma \ref{lem:I_3}.
Moreover, if there were more than two bisections meeting the same fibre component of the $\IK_4$,
then we could build a $\IK_0^*$ divisor intersected by some  fibre component
with multiplicity one, contradiction.
Hence each fibre component is met by at most two bisections.

If the fibration given by $|2\IK_4|$ has 5 bisections, 
then one can systematically go through all possible configurations 
to confirm the claim.
In fact, for each configuration, one easily finds a suitable subconfiguration of rank exceeding 10.
But this is the Picard number of any Enriques surface, regardless of the characteristic (see \cite{BM76}),
contradiction.
%
%
%
%
\end{proof}

After these preparations we can finally complete the proof of the bound.

\subsection*{Proof of Theorem \ref{thm}}
In order to derive a contradiction, 
we assume that $H^2 >12d^2$ and  $\#\Gamma>12$. 
Recall that $\Gamma$ is hyperbolic by Cor.~\ref{cor:para}.

Lemmas~\ref{lem:red2},~\ref{lem:I_3} yield that $\Gamma$ contains an $\IK_4$ configuration. 
From Lemma~\ref{lem:types} we infer that $\Gamma$ consists of fiber components and bisections of $|2\IK_4|$.
Lemma~\ref{lem:I_4-5bisections} implies that  $\Gamma$ contains at most four bisections.

If  $\Gamma$ contains  at least 10 fiber components, then
Lemma~\ref{lem:red1-3sect} combined with Lemma~\ref{lem:conf76} leads to a contradiction.
Thus the assumption $\#\Gamma>12$ implies that $\Gamma$ contains exactly nine fiber components 
and exactly four bisections. 
The classification in \cite{OS} gives the following  configurations:

\subsection{
$\IK_4+\IK_5$} 

Together with one bisection these curves support a divisor of type $\IK_3$, leading 
back to Lemma \ref{lem:I_3},
 or of type $\IK_2^*$
met by some fibre component with multiplicity one, contradiction.

\subsection{
$\IK_4+\IK_0^*$}

Here each bisection connects with $\IK_0^*$ to another $\IK_4$,
so one of these has at least 5 bisections, which is impossible by Lemma~\ref{lem:I_4-5bisections}.

\subsection{
$\IK_4+A_3+A_1+A_1$}

The root lattices could be realized inside a single $\IK_1^*$
or inside $\IK_4+\IK_2+\IK_2$.
In the former case, any bisection splits into section and two-torsion section on the K3 cover,
so we derive Kondo's type II surfaces -- with exactly 12 smooth rational curves by \cite{Kondo}, \cite{Martin}.
In the latter case, the given curves together with any bisection form a $\QQ$-basis of $\Num(Y)$,
and one can easily check that any  configuration  either leads back to a previous case
or leaves no room for further smooth rational bisections fitting our scheme imposed by \eqref{eq:C}.

\subsection{
$\IK_4+D_5$}

This case either gives Kondo's family of type II again
or leaves no room for more than one smooth rational bisection. 

\subsection{
$\IK_4+\IK_4+A_1$}

Here any bisection splits into section and two-torsion section on the K3 cover,
so the bisection does not meet the $A_1$ summand (which embeds into an $\IK_2$ fibre).
Hence all computations can be carried out in a lattice of rank at most $9$.
For each possible configuration, this quickly leads to a contradiction. 
\qed

\section{Enriques surfaces with 12 rational curves of even degree} 
\label{ss:even_d}

Proving Proposition \ref{prop} (i) amounts to the following:

\begin{prop}
\label{prop:even_d}
Let $d>2$ and $h\geq 3d$.
Then there is an Enriques surface of degree $2h$
containing 12 rational curves of degree $2d$.
\end{prop}

\begin{proof}
Let $Y$ be a general Enriques surface
such that the following hold:
\begin{itemize}
\item
it contains no smooth rational curve;
\item
each genus one fibration on $Y$  has exactly 12 reduced singular fibers, all of which are nodal cubics;
\item
for each genus one fibration on $Y$ both half-fibers  are irreducible, smooth elliptic curves.  
\end{itemize}
The existence of such a surface can be shown with help of 
the relation between deformations of unnodal Enriques surfaces and deformations of
their Jacobians (see \cite[Remark~5.6]{Martin2}).

Recall that there exist   half-pencils $E, E'$ of genus one fibrations on $Y$ such that $E.E'=1$,
 (this follows from   \cite[Thm~6.1.10]{Dolgachev-Kondo};  
 here the assumption that $Y$ is unnodal 
 is of importance only when $\mbox{char}(k)=2$).

As in the K3 case, we want to work with a sublattice
\[
L = \begin{pmatrix}0 & d\\ d& c\end{pmatrix}
\hookrightarrow\Num(Y) 
\cong U + E_8
\]
for some $c\in\{0,\hdots,2d-2\}$.

Here we set up $L$ as follows:
We consider $U\subset\Num(Y)$, that is generated by the half-pencils $E, E'$  
and pick $D'\in U^\perp$ with $D'^2 = c-2d$.
Consider the divisor
\[
D = E + dE' + D' \;\;\; \text{ with } \;\;\; D^2 = c.
\]
By Riemann--Roch, $D$ is effective (since $D.E'=1$ and $D^2 \geq 0$).
Moreover, by assumption,
all singular fibres of $|2E|$ have type $\IK_1$,
so there are 12 in number,
none of which is multiple.	
Consider the divisor 
\[
H = NE + D \;\;\; (N\in\NN).
\]
We claim that $H$ is very ample if $N\geq 3$.
Indeed, we have
\begin{itemize}
\item
$H^2 = 2Nd+c\geq 6N$;
\item
$H.C>0$ for any curve $C\subset Y$,
since $C$ is multisection for $|2E|$ or $|2E'|$, so
$H.C\geq (E+E').C>0$;
%
\item
$H.E''>2$ for any half-pencil $E''$ on $Y$, since either $2E''\sim 2E$ and $H.E''=H.E=d>2$
or $E''$ is a multisection of $|2E|$, so $H.E''\geq N$.
\end{itemize}
It follows from 
Criterion \ref{crit}
that for any $h\geq 9$, we find $N$ and $c$ as above such that $H$ is very ample
with $H^2=2h$. By construction, $(Y, H)$ contains  rational 12 curves of degree $2d$
(the singular fibres of $|2E|$).
\end{proof}

\section{Enriques surfaces with 12 smooth rational curves of fixed degree}
\label{ss:all_d}

In this section we prove Proposition \ref{prop} (ii).
This is equivalent to the following:

\begin{prop}
\label{prop:all_d}
Assume that the characteristic of $k$ differs from $2$.
Let $d>2$ and $h\geq 9$ such that $d\mid h$.
Then there is an Enriques surface of degree $2h$
containing 12 smooth rational curves of degree $d$.
\end{prop}

To prove the proposition, 
we aim at constructing a 3-dimensional family of Enriques surfaces $Y$
with a genus one fibration with 6 fibres of type $\IK_2$
and a rational bisection $B$ with $B^2=2$
which meets each reducible fibre in both components.

\subsection{Base change construction}
Start with the rational elliptic surface $S$ which will feature as Jacobian of $Y$.
This takes the shape
\[
S: \;\;\; y^2 = x(x-f)(x-g), \;\;\; f,g\in k[t], \; \deg(f)=\deg(g)=2.
\]
We want to apply the base change construction from \cite{HS}, 
i.e.\ endow some quadratic base change of $S$ (an elliptic K3 surface $X$)
with a section $P$ which is anti-invariant in $\MW(X)$ for the deck transformation $\imath$
of the double cover $X\dasharrow S$.
Thus  we are led to work on the quadratic twist
\begin{eqnarray}
\label{eq:S'}
S':\;\;\; qy^2 = x(x-f)(x-g)
\end{eqnarray}
at some quadratic non-square polynomial $q\in k[t]$.
We need to endow $S'$ with a section $P'$ of height $3$
which meets all reducible fibres (6 $\IK_2$ and 2 $\IK_0^*$) non-trivially.
That is, $P'.O'=2$ with intersection points at another quadratic polynomial $w\in k[t]$.
In the above affine model \eqref{eq:S'}, the section $P'$ takes the shape
\begin{eqnarray}
\label{eq:P'}
\;\;\;\;\;
P'=\left(\dfrac{fgr}{w^2}, \dfrac{fg(f-g)r'}{w^3}\right),\; \;\;\; r,r'\in k[t], \; \deg(r)=\deg(r')=2.
\end{eqnarray}
where the $\IK_2$ fibres are located at the zeroes of $fg(f-g)$
(so the $y$-coordinate of $P'$ vanishes at all of them as should be).
We now make the following convenient normalizations and choices:
\begin{itemize}
\item
$w$ is monic (by absorbing the top coefficient into the nominator);
\item
$f$ is monic (by rescaling $x,y$)
\item
the zeroes of $f-g$ are $t=0, \infty$ by M\"obius transformation,
i.e.\ the top and bottom coefficients agree:
\[
f = t^2 + bt + c^2; \;\;\; g = t^2 + b't + c^2.
\]
\item
$r$ is a square with zero at $t=-1$;
the second choice amounts to a M\"obius transformation again
while the first is one of a few natural choices on elliptic surfaces with  two-torsion in $\MW$, cf. \cite[\S 8.1.4]{S-Q-hom}.
\end{itemize}
For $P'$ to meet the node of $\IK_2$ fibres at $t=0, \infty$ of the Weierstrass model \eqref{eq:S'}
(so it meets the non-identity component of the Kodaira--N\'eron model),
we read off the top coefficient of $r$ and subsequently the bottom coefficient of $w$:
\[
r = (t+1)^2, \;\;\; w = t^2 + d't + c.
\]
This yields $r'=(t+1)r''$,
and it remains to choose the coefficients $b, b', c, d'$ in such a way that 
substituting the $x$-coordinate of $P'$ into the RHS of \eqref{eq:S'}
gives a square $r''^2$ next to the other obvious squares.
This amounts to computing the discriminant of a degree 4 polynomial
and has, up to symmetry in $b$ and $b'$, the unique solution
\[
d' = 2d-1-c, \;\;\; b' = d^2-2c.
\]
One then verifies that generically this exactly gives the configuration of singular fibres and section
which we are aiming for.

\subsection{Enriques quotient}

The quadratic base change of $S$ (and $S'$) ramified at the zeroes of $q$ generically
gives the announced elliptic K3 surface $X$ with 12 $\IK_2$ fibres
and section $P$ of height $6$ pulling back from $P'$.
By construction, $P'$ is anti-invariant in $\MW(X)$ for the deck transformation $\imath$,
so composing $\imath$ with translation by $P$ gives a fixed point free involution $\jmath$
(since $P$ meets non-trivial two-torsion points at the fixed fibres,
corresponding to the non-identity components of the $\IK_0^*$ fibres on $S'$).
Hence $Y = X/\langle\jmath\rangle$ is an Enriques surface,
and $O, P$ map down to a rational bisection $B$ which meets each $\IK_2$ fibre
in both components and which has two nodes at the zeroes of $w$, i.e.\ $B^2=2$. 
(Note that $P.O=4$ and $(P+O)^2 = 4$.)

\subsection{Degree $d$ smooth rational curves on $Y$}

Consider the induced half-pencil $E$  on $Y$.
Instead of working with the bisection $B$, we note that the isotropic divisor
\[
E' = B-E
\]
is effective by Riemann--Roch (since $E.E'=1$).
We claim that $|2E'|$ has no base locus,
for otherwise there would be a $(-2)$-curve $C$ such that 
\[
E'.C<0 \;\;\; \text{ and } \;\;\; E''=E'+(E'.C)C>0.
\]
But then $E.E'=1$ implies, since $E$ is nef, that $E.C=0$ or $1$
and same with interchanged values for $E.E''$.
In the former case, we derive the contradiction
\[
0<B.C = (E+E').C < 0.
\]
In the latter case, we have $E.E''=0$, so $E\equiv E''$ and $C$ is a bisection of $|2E|$.
But then it splits on the K3 cover $X$ into two disjoint 
sections always meeting opposite components of the $\IK_2$ fibres,
 and as in Section \ref{ss:3},
the height of one section relative to the other returns $h=4-12\cdot\frac 12=-2$ which is impossible.

To conclude the proof of Proposition \ref{prop:all_d}, consider the divisor
\[
NE+dE'
\]
which is easily verified, using Criterion \ref{crit}, to be very ample for $N,d\geq 3$.
This leads to the claimed range of polarizations, and to 12 smooth rational curves of degree $d$ 
-- the components of the six $\IK_2$ fibres.
\qed

 \subsection{Odd degree curves in characteristic 2}
 We conclude this section with an indication
 why characteristic $2$ can be rather special:
 
 \begin{prop}
 \label{prop:p=2}
 Let $d$ be odd. Let $Y$ be an Enriques surface with smooth K3 cover
 over a field of characteristic $2$, endowed with a polarization of degree $H^2\gg 0$.
 Then 
 \[
 r_d<11.
 \]
 \end{prop}
 
 \begin{proof}
 Assume first that $r_d=12$. Then Lemma \ref{lem:para} implies 
 that $\Gamma$ cannot be hyperbolic, so by Corollary \ref{cor:para},
 it is parabolic, and all curves in $\Gamma$ form fibre components (of the same degree $d$)
 of a genus one fibration on $Y$.
 Necessarily, this is semi-stable, and all reduced fibres have the same degree,
 hence the same type $\IK_{2n}$.
 (If there were an odd number of components, then $H.F$ would be odd for a general fibre $F$,
 but writing $H$ as a $\ZZ$-linear combination of fibre components and multisections
 of course leads to $H.F$ being even.)
 Along the same lines, the support of any multiple fibre may only be smooth or of type $\IK_{n}$.
This allows only for the following configurations:
\[
6\times \IK_2, \;\;\; 5\times \IK_2 + 2\times \IK_1, \;\;\; 2\times \IK_4 + 2\times \IK_2.
\]
Note that the first configuration is used in Proposition \ref{prop:all_d} 
-- outside characteristic $2$!
Indeed, so far our argument has been independent of the characteristic.
But then the last two configurations involve two multiple fibres
which is impossible in characteristic $2$.
Moreover
the first and the last configuration do not exist
on a rational elliptic surface in characteristic $2$ by \cite[Thm.\ 8.9]{SS-MWL}.
Hence $r_d<12$.

We continue by assuming that $r_d=11$.
As before, all curves in $\Gamma$ are supported on the fibres of a genus one fibration,
so there are two possibilities:
\begin{enumerate}
\item
either all singular fibres are multiplicative, and exactly one of the 12 rational fibre components
is not contained in $\Gamma$,
\item
or exactly one singular fibre is additive (without wild ramification),
and all rational fibre components are contained in $\Gamma$.
\end{enumerate}
The second case is easy to rule out, since the only additive fibre types
without wild ramification in characteristic $2$ are $\IV, \IV^*$ by \cite{SSc}.
But then the first has three fibre components, contradicting what we have seen above,
while the second has degree $12d$ which is too large to be met by $\IK_n$ fibres (supported on $\Gamma$)
by inspection of the Euler--Poincar\'e characteristic.

In the first case, one easily checks that the multiple fibre, if singular,
is supported on $\Gamma$, leaving the following configurations of 11 rational curves in $\Gamma$:
\[
2\times\IK_4 + \IK_2 (\text{multiple}) + A_1,\;\;\;\; 
5 \times \IK_2 + \IK_1 (\text{multiple}).
\]
The first configuration has $A_1$ inside $\IK_2$, so this cannot exist as pointed out before.
Similarly, the second configuration generally forces a $2$-torsion section upon the Jacobian by \cite{OS}.
But then, by \cite[Cor.\ 8.32]{SS-MWL}, there is an additive fibre in characteristic $2$, 
ruling out this configuration as well.
 \end{proof}

%

\section{An Enriques surface with 13 smooth rational curves of small degree}

In this section we demonstrate that the bound for $H^2$ in Theorem \ref{thm}
cannot be improved beyond $H^2>12d^2-2$ as soon as $d>1$.
To this end, we prove the following:

\begin{prop}
\label{prop:13}
Let $Y$ be the Enriques surface of Kondo's type VII
over a field of characteristic $\neq 2,5$.
Let $d\in\NN$.
Then $Y$ admits a polarization $H$ of degree $H^2=12d^2-2$
and 12 smooth rational curves of degree $d$ as well as one of degree $2$.
\end{prop}

\begin{proof}
As in \ref{ss:VII},
$Y$ is endowed a genus one fibration with reducible fibres $\IK_6, \IK_3$ 
(ramified, corresponding to the half-pencil $E$)
and $\IK_2$ (or $\III$ in characteristic $3$).
Fix a component $\Theta$ of the $\IK_2$ resp.\ $\III$ fibre
and fix those three smooth rational bisections $B_1, B_2, B_3$ which do not meet $\Theta$
(and are pairwise disjoint).
We postulate that
\begin{itemize}
\item
$B_1, B_2, B_3$ and all components of the $\IK_6$ and $\IK_3$ fibres have degree $d$;
\item
$\deg(\Theta)=2$.
\end{itemize}
These degrees determine the 
polarization $H$ uniquely as
\[
H = (3d-2)E + d(B_1+B_2+B_3) + \Theta'
\]
where $\Theta'$ is the component of the $\IK_2$ resp.\ $\III$ fibre other than $\Theta$.
We claim that $H$ is very ample.
To verify this by Criterion \ref{crit},
one computes directly that
\begin{itemize}
\item
$H^2=12d^2-2\geq 10$;
\item
$H.C>0$ for any smooth rational curve on $Y$ (the other degrees 
appearing are $4d-2$ and $6d-2$);
\item
$H.C>0$ for any other curve in $Y$, since this is either a multisection of $|2E|$,
so $H.C\geq (3d-2)E.C\geq 3d-2$, or a (half)-fibre, so $H.C=d(B_1+B_2+B_3).C\geq 3d$.
\end{itemize}
It remains to discuss the case where $E'$ is a half-pencil $\not\equiv E$.
Consider the half-pencils 
$$E''_i=B_i+\Theta'
\;\; \text{ with bisections } \;\; B_j, B_k \;\; (\{i,j,k\}=\{1,2,3\}).
$$
Then either $E'\equiv E''_i$ for some $i$, so $H.E'\geq (B_j+B_k).\Theta'=4d$, or 
$E'$ is a multisection for $|2E|$ and all $|2E''_i|$, so
\begin{eqnarray}
\label{eq:E'}
H.E'\geq ((3d-2)E+B_i+\Theta').E'\geq 3d-1,
\end{eqnarray}
with equality for each $i$ if and only if $E'.E=1, E'.\Theta'=1$ and  $E'.B_i=0$.
Now $\Theta'$ meets each smooth rational curve on $Y$ with even multiplicity
by \cite[Fig.~7.7]{Kondo},
and the possible half-pencils are classified in \cite[Table 2]{Kondo}:
\begin{itemize}
\item
$E'\equiv\IK_2$ or $\IK_3$, so $E'.\Theta'\in 2\ZZ$;
\item
$E'\equiv\frac 12 \IK_5$ (for either of the two $\IK_5$ fibres) or $\frac 12 \IK_9$.
Since $E'.B_i=0$, each $B_i$ is contained in some fibre of $|2E'|$.
But then one of the fibres contains two of the $B_i$, so $\Theta'.E'\geq \Theta'.(B_i+B_j)/2=2$.
\end{itemize}
In either case, \eqref{eq:E'} thus improves to $H.E'\geq 3d\geq 3$.
Hence Criterion \ref{crit} applies to show that $H$ is very ample.
By construction, $Y$ contains the smooth rational curves of the given degrees (relative to $H$).
\end{proof}

\subsection*{Acknowledgement}
We thank the referee for helpful comments.

\end{document}